\documentclass{article}
\usepackage{
amsmath,
amsfonts,
latexsym, 
amsrefs,
amssymb
}

\usepackage{bm}

\usepackage{tocloft}

\usepackage{cite}

\usepackage{pstricks}

\usepackage{subcaption}

\usepackage{tikz}
\usepackage{tikz,fullpage}
\usetikzlibrary{arrows,%
                petri,%
                topaths, automata}%

\usepackage{graphicx}
\DeclareMathOperator{\Cov}{Cov}
\newcommand{\phe}{\varphi}
\newcommand{\ds}{\displaystyle}

\DeclareMathOperator{\Ima}{Im}
\newcommand{\Proba}{\mathbb{P}}
\DeclareMathOperator{\Rea}{Re}
\newcommand{\dsum}{\displaystyle\sum}
\usepackage{bbm, dsfont}

\usepackage{enumerate}
\usepackage{mathrsfs}
\usepackage{wrapfig}

\usepackage{color}

\numberwithin{equation}{section}

\newcommand{\eps}{\varepsilon}

\newcommand{\rd}{{\rm d}}

\newcommand{\be}{\begin{equation}}
\newcommand{\ee}{\end{equation}}

\usepackage{amsmath} 
\usepackage{amssymb}
\usepackage{amsthm}

\renewcommand{\epsilon}{\varepsilon}
\renewcommand{\leq}{\leqslant}
\renewcommand{\geq}{\geqslant}

\newcommand{\E}{\mathbb{E}}
\newcommand{\R}{\mathbb{R}}
\newcommand{\C}{\mathbb{C}}
\newcommand{\N}{\mathbb{N}}

\theoremstyle{plain} 
\newtheorem{theorem}{Theorem}[section]
\newtheorem*{theorem*}{Theorem}
\newtheorem{lemma}[theorem]{Lemma}
\newtheorem*{lemma*}{Lemma}
\newtheorem{corollary}[theorem]{Corollary}
\newtheorem*{corollary*}{Corollary}
\newtheorem{proposition}[theorem]{Proposition}
\newtheorem*{proposition*}{Proposition}

\newtheorem*{definition*}{Definition}

\newtheorem*{example*}{Example}
\newtheorem{remark}[theorem]{Remark}

\newtheorem*{remark*}{Remark}
\newtheorem*{remarks*}{Remarks}

\makeatletter
\renewcommand{\subsection}{\@startsection
{subsection}
{2}
{0mm}
{-\baselineskip}
{0 \baselineskip}
{\normalfont\bf\itshape}} 
\makeatother

\usepackage{dsfont}
\usepackage{stmaryrd}

\newcommand{\Hf}{\dfrac{1}{2}}

\def\@empty{}

\def\author#1{\par
    {\centering{\authorfont#1}\par\vspace*{0.05in}}
}

\def\titlefont{\fontsize{13}{15}\bfseries\boldmath\selectfont\centering{}}
\def\authorfont{\fontsize{13}{15}}
\def\abstractfont{\fontsize{8}{10}}

\let\affiliationfont\rhfont

\def\address#1{\par
    {\centering{\affiliationfont#1\par}}\par\vspace*{11pt}
}

\def\body{
\setcounter{footnote}{0}
\def\thefootnote{\alph{footnote}}
\def\@makefnmark{{$^{\rm \@thefnmark}$}}
}

\def\title#1{
    \thispagestyle{plain}
    \vspace*{-14pt}
    \vskip 79pt
    {\centering{\titlefont #1\par}}%
    \vskip 1em
}

\renewenvironment{abstract}{\par%
    \vspace*{6pt}\noindent 
    \abstractfont
    \noindent\leftskip10pt\rightskip10pt
}{%
  \par}

\usepackage{tocloft}

\makeatletter
\renewcommand{\section}{\@startsection
{section}
{1}
{0mm}
{-2\baselineskip}
{1\baselineskip}
{\normalfont\large\scshape\centering}} 
\makeatother

\newcommand{\hf}{\frac{1}{2}}

\begin{document}

~\vspace{-1.4cm}

\title{Brownian behaviour of the Riemann zeta function around the critical line}

\vspace{1cm}
\noindent

\begin{minipage}[b]{0.3\textwidth}
\hspace{3cm}
 
 \end{minipage}
 \begin{minipage}[b]{0.3\textwidth}
 \author{Louis Vassaux}

\address{INRIA, DMA/ENS
\\
louis.vassaux@ens.psl.eu}
 \end{minipage}

\begin{minipage}[b]{0.3\textwidth}

 \end{minipage}

\begin{abstract}
We establish a Brownian extension to Selberg's central limit theorem for the Riemann zeta function.  This implies various limiting distributions for $\zeta$, including an analogue of the reflection principle for the maximum of the Brownian motion: as $T$ diverges,  for any $u>0$ we have
\[
\frac{1}{T}\cdot {\rm meas}\Big\{0\leq t\leq T:\max_{\sigma\geq \tfrac{1}{2}}\log|\zeta(\sigma+i t)|\geq u \sqrt{\tfrac{1}{2}\log \log T} \Big\}\to 2\int_u^{\infty} \frac{e^{-\frac{x^2}{2}}}{\sqrt{2\pi}}\rd x.
\]
\end{abstract}

\vspace{-0.3cm}

\tableofcontents

\vspace{0.2cm}

\section{Introduction}

\subsection{Statistics on $\zeta$.}\ 
Selberg's central limit theorem  \cite{Sel1946} states that for $\tau$ is chosen uniformly at random from $[0, T]$,

\begin{equation}
    \label{clt}
    \dfrac{1}{\sqrt{\log \log T}}\, \log \zeta\big(\hf + i \tau\big) \overset{\text{law}}{\longrightarrow} \mathscr{N}_{\mathbb{C}}
\end{equation}
as $T \rightarrow + \infty$,  where $\mathscr{N}_{\mathbb{C}}$ is a standard complex Gaussian random variable, i.e.   with independent normal real and imaginary parts of variance $\tfrac{1}{2}$.  A simple proof for $\log |\zeta|$ can be found in \cite{RadSou2017}. 

This theorem was the first major result on the statistical behaviour of $\zeta$ around the critical line; since then,  many advances in probabilistic number theory involve new statistics for $\zeta$.  A prominent example is Montgomery's pair correlation \cite{pair}, which conjecturally identifies the local spacings between high zeros of Riemann zeta function with the scaling limit of random matrices from  the Gaussian unitary ensemble.

More recently,  the  Fyodorov-Hiary-Keating conjecture \cite{FyoHiaKea2012,fyodorov2014freezing} proposed to extend this analogy to extreme values statistics,  through precise estimates for the maximum of the Riemann zeta function over short intervals on the critical line.  In particular,  these papers conjectured that,  if $\tau \in [T, 2T]$ is chosen uniformly at random, 
\begin{equation}\label{FHK}
    \max_{|h| \leq 1} |\zeta(\hf + i \tau + i h)| \asymp \frac{\log T}{(\log \log T)^{\frac{3}{4}}}
\end{equation}
meaning that the ratio between the two sides is tight as $T \to \infty$.
  After initial progress in \cite{Naj16,ArgBelBouRadSou2016,Har2019},  these estimates on this ratio and its universal tail asymptotics were proved in 
\cite{arguin2020fyodorov1,arguin2020fyodorov2}. 

One natural question we consider in this paper is about the maximum of $\zeta$ over horizontal intervals. Our main result is to control the behaviour of $\zeta$ in intervals of the form $[\hf+ i \tau, \hf + \eta + i \tau]$ by exhibiting a form of functional tightness. This - for example - will allow us to show that for large $T$
\begin{equation}\label{eqn:maxPr}
    \frac{1}{\sqrt{\log \log T}}\max_{\sigma \geq \tfrac{1}{2}} \log|\zeta( \sigma + i \tau)|\overset{\text{law}}{\longrightarrow} |\mathscr{N}(0,\hf)|,
\end{equation}
where, again, $\tau \in [0,T]$ is chosen uniformly at random.  In particular, while the maximum (\ref{FHK}) over vertical short intervals exhibits a subtle, larger scaling than (\ref{clt}),
the  size order for the maximum over horizontal intervals agrees with the normalization in Selberg's central limit theorem, but the limiting random variable is the absolute value of a Gaussian instead of a Gaussian,
mirroring the reflection principle for the Brownian motion.

\subsection{Main result.}\ 
Our main statement is as follows.

\begin{theorem}
\label{mainthm}
    Let  $T > 10$ an $\tau$ be a uniform random variable on $[0, T]$. Define
\begin{displaymath}
Z^{(T)}:
\left.
  \begin{array}{rcl}
    [0, 1] & \longrightarrow &\C \\
    \alpha & \longmapsto & \dfrac{1}{\sqrt{\log \log T}} \log \zeta (\hf + \dfrac{1}{(\log T)^{\alpha}} + i \tau). \\
  \end{array}
\right.
\end{displaymath}
Then, as $T \rightarrow \infty$, $Z^{(T)}$ converges in law,  for the topology of uniform convergence in $\mathcal{C}^0([0,1], \C)$, to a standard complex Brownian motion $B$, i.e. 
$B=\tfrac{1}{\sqrt{2}}(B_1+iB_2)$ where $B_1$ and $B_2$ are independent,  standard real Brownian motions. 
\end{theorem}
\noindent The above convergence means that for any  $F: \mathcal{C}^0 \to \R$ a bounded functional which is continuous for the ${\rm L}^\infty([0,1])$ norm,  we have
\begin{equation}
    \E[F(Z^{(T)})] \underset{T \to + \infty}{\longrightarrow} \E[F(B)].
\end{equation}
\begin{remark} \label{exttoline}
  The theorem above is stated for $\alpha \leq 1$ because the process doesn't exhibit any interesting behaviour beyond this point: if we wish to consider $\alpha \in [0, \infty)$, the limit process would follow the law of 
\begin{equation} \label{defbtilde}
    {\tilde B}_\alpha  = \begin{cases}
    B_\alpha & \text{ if } \alpha \leq 1,\\
    B_1 & \text{ if } \alpha \geq 1.
\end{cases}
\end{equation}
Even in this setting ($\alpha \in [0, \infty)$) we have convergence in distribution of the sequence of functions in $\mathcal{C}^0([0, \infty), \C)$ (for the topology of uniform convergence). The proof is the same,  except some minor changes to the proof of Theorem \ref{kolmo}.
\end{remark}

Theorem \ref{mainthm} implies that  for any 
 $0 \leq \alpha_1, \ldots, \alpha_n \leq 1$ the joint limit distribution of $(Z^{(T)}(\alpha_i))_{1 \leq i \leq n}$ is that of a Gaussian vector (with an appropriate correlation structure, described in Theorem \ref{fidi}). However, this finite-dimensional limit is already well-understood: Our proof of Theorem \ref{fidi} is only a repurposing of the methods from \cite{bourgade}. Our main contribution is instead the fact that the sequence of random functions $(Z^{(T)})_{T \geq 0}$ is tight, which is not \text{a priori} obvious and is necessary to deduce corollaries such as (\ref{eqn:maxPr}).

Finally, we note that for random unitary matrices, there is an analogous result to the finite dimensional convergence,  Theorem \ref{fidi}, also shown in \cite{bourgade}. One might wonder whether we can also generalise this to a functional convergence towards Brownian motion, similarly to Theorem \ref{mainthm}. More specifically, if $U_n$ is a random, Haar-distributed,  $n\times n$ unitary matrix,,  the relevant analogue of $Z^{(T)}$ to consider would be
\begin{displaymath}
\mathcal{Z}^{(n)}:
\left.
  \begin{array}{rcl}
    [0, 1] & \longrightarrow &\C \\
    \alpha & \longmapsto & \dfrac{1}{\sqrt{\log n}}\log\det(\exp(n^{- \alpha})I_n - U_n).\\
  \end{array}
\right.
\end{displaymath}
It seems likely that this process converges in distribution to a standard, complex Brownian motion as $n\to\infty$.

\subsection{Consequences.}\  Convergence in law to Brownian motion gives a few corollaries regarding the global behaviour of the Riemann zeta function in the critical strip; to the author's knowledge, these corollaries are novel.
Similar statements hold if we replace $\log |\zeta|$ with $\Ima \log \zeta$.

\begin{corollary}[Reflection principle] \label{refl}
 The convergence (\ref{eqn:maxPr}) holds.
\end{corollary}

\begin{proof}
    It follows from Remark \ref{exttoline}, and the reflection principle for Brownian motion (see \cite[Proposition III.3.7]{revuzyor}), that
    \begin{equation}
        \sup_{\sigma \in [\tfrac{1}{2}, \tfrac{3}{2}]} \frac{1}{\sqrt{\log \log T}} \log |\zeta(\sigma + i \tau)| \underset{law}{\longrightarrow} |\mathscr{N}(0, \frac{1}{2})|.
    \end{equation}
    Furthermore, if $\sigma \geq 3/2$, $|\zeta(\sigma + i \tau)| \leq 2$; this proves the corollary.
\end{proof}

The next result gives the distribution of the logarithmic measure of $\sigma$ such that $\log|\zeta(\sigma+ i\tau)|\geq 0$. 
\begin{corollary}[Arcsine law]
    For $T > 10$, define the probability measure $\mu_T$ over $\left[\dfrac{1}{\log T}, 1\right]$ by

    \begin{equation}
        d \mu_T (\eta) = \dfrac{1}{\eta\log \log T}d \eta.
    \end{equation}
    Then, if $\tau \in [0, T]$ is chosen uniformly, the distribution of
$$M_T = \mu_T \{ \eta \in [\dfrac{1}{\log T}, 1], |\zeta(\hf + \eta + i \tau)| \geq 1 \}$$
converges weakly to an arcsine law i.e. $\Proba(M_T \leq y) \underset{T \to \infty}{\longrightarrow} \dfrac{2}{\pi} \arcsin{\sqrt{y}}$.
\end{corollary}

\begin{proof}
    Applying the arcsine law for Brownian motion (see \cite[Theorem VI.2.7]{revuzyor}), we know that 
    \begin{equation}
        \tilde{M_T} = \lambda \{ \alpha \in [0, 1], \log |\zeta(\hf + \frac{1}{(\log T)^\alpha} + i \tau)| \geq 0 \}
    \end{equation}
    converges weakly to an arcsine law. However, if $\alpha \in [0, 1]$ is chosen uniformly, $\dfrac{1}{(\log T)^\alpha}$ follows the law $\mu_T$; the corollary follows.
\end{proof}

A third consequence of our theorem is a result quantifying how abnormally large the maximum around $\tfrac{1}{2}+ i\tau$ (approached horizontally) becomes.

\begin{corollary}[Law of the iterated logarithm]
    If $\alpha \in [0,1]$, set
    \begin{equation}
        S_t(\alpha) = \sup_{0 \leq \beta \leq \alpha} \left|\log |\zeta(\hf + \frac{1}{(\log T)^{\beta}} + i t)|\right|.
    \end{equation}
    Then, if $\tau \in [0, T]$ is chosen uniformly at random, the sequence of functions $\left(\dfrac{S_{\tau}}{\sqrt{\log \log T}}\right)_{T \geq 10}$ converges in distribution to a random function $S$ verifying 

    \begin{equation} \label{itlogformula}
        \lim_{\alpha \to 0} \frac{S(\alpha)}{\sqrt{\alpha \log \log \frac{1}{\alpha}}} = 1 \text{ a.s.}
    \end{equation}

\end{corollary}

This originates from the so-called iterated logarithm law for Brownian motion, which follows from a similar law for random walks \cite{itlog}. The usual statement is that, for a real Wiener process $(X_t)_{t \geq 0}$,

\begin{equation}
    \limsup_{t \to + \infty} \frac{X_t}{\sqrt{2t \log \log t}} = 1 \text{ a.s.}
\end{equation}
but, noting that $(t \mapsto t X_{\frac{1}{t}})$ is also a Wiener process, an analogous law follows for $t \to 0$.

\begin{proof}
    Let $F: \mathcal{C}^0([0,1], \R) \to \R$ be a continuous bounded functional. The functional
    \begin{displaymath}
G:
\left.
  \begin{array}{rcl}
    \mathcal{C}^0([0,1], \C) & \longrightarrow &\R \\
    Z & \longmapsto & F\left(\left( \alpha \mapsto \ds\sup_{0 \leq \beta \leq \alpha} \Rea  Z(\beta)\right)\right)\\
  \end{array}
\right.
\end{displaymath}
is also continuous and bounded. Thus,

\begin{equation}
    \begin{split}
        \E\left[F \left(\frac{1}{\sqrt{\log \log T}}S_{\tau}\right)\right] & = \E[G(Z^{(T)})] \underset{T \to +\infty}{\longrightarrow} \E[G(B)] = \E[F(S)],
    \end{split}
\end{equation}
where $B$ is a complex Brownian motion and $S(\alpha) = \ds\sup_{0\leq \beta \leq \alpha}\Rea B_\beta$.

Then, $(\ref{itlogformula})$ follows from the aforementioned iterated logarithm law for Brownian motion.
\end{proof}

Another corollary of our functional convergence is the limiting occupation measure for $\log |\zeta|$, on horizontal lines. More precisely, 
 for $t > 0$, let $L^{(T)}_t$ be the local time of $\Rea Z^{(T)}$ i.e. the (almost everywhere) unique function $L_t : \R \rightarrow \R_+$ such that, for any function $\phe \in \mathcal{C}^0(\R, \R)$,
    \begin{equation}
        \int_{0}^t \phe(\Rea Z^{(T)}(u)) du = \int_{\R} \phe(v) L_t(v) dv.
    \end{equation}

\begin{corollary}
\label{localtime}
The process $L^{(T)}_t$ converges weakly to the local time $L^{(\infty)}_t$ of Brownian motion, in the following sense: if $f, \phe \in C_b^0(\R)$ are bounded continuous functions,
    \begin{equation}
        \E[f(\langle L_t^{(T)}, \phe \rangle)] \underset{T \to \infty}{\longrightarrow} \E[f(\langle L_t^{(\infty)}, \phe \rangle)].
    \end{equation}
\end{corollary}

\begin{proof}

    Let $\phe \in \mathcal{C}^0_b(\R)$. Then
\begin{multline*}
        \E \left[ f \left(\int_{\R} \phe(v) L_t^{(T)}(v) dv \right) \right] = \E \left[f \left(  \int_0^t \phe \circ \Rea Z^{(T)} (u) du \right)\right] 
         \underset{T \to \infty}{\longrightarrow} \E \left[ f \left( \int_0^t \phe (B_u)du \right) \right]\\
         = \E \left[ f \left( \int_{\R} \phe(v) L^{(\infty)}_t(v) dv \right) \right],
\end{multline*}
as desired.
\end{proof}

This result grants us insight into the distribution of values of $\log \zeta$. For instance:

\begin{corollary}
    Let $N > 0, \eps > 0$. Then, there exists $T_0(N, \eps) > 0$ such that, if $T \geq T_0$ and $\tau \in [0, T]$ is uniformly chosen, 
    \begin{equation}
        \Proba(\log |\zeta( \sigma + i\tau)| \text{ changes sign at least N times for } \sigma \in [\tfrac{1}{2},\tfrac{3}{2}]) \geq 1 - \eps.
    \end{equation}
\end{corollary}

\begin{proof}
    By monotone convergence, it is sufficient to show that, for any $\eps > 0$, $\eta > 0$, there exists $T_0(\eta, \eps)$ such that, for $T \geq T_0$, 
    \begin{equation}
    \label{useltimelemma}
        \Proba(\Rea Z^{(T)} \text{ changes sign over }[0, \eta]) \geq 1-\eps.
    \end{equation}
    To show this, consider $\phe(x) = x^+ \land 1$. By Portmanteau's theorem,
    \begin{equation}
        \liminf_{T \to \infty} \Proba(\langle L_\eta^{(T)}, \phe \rangle > 0) \geq \Proba(\langle L^{(\infty)}_{\eta}, \phe \rangle > 0) = 1
    \end{equation}
    and so, for large enough $T$, $\Rea Z^{(T)} > 0$ at some point in $[0, \eta]$ with probability at least $1 - \eps$. 
    The same argument applied to $\phe(x) = x^- \land 1$ shows (\ref{useltimelemma}), and the corollary.
\end{proof}

\subsection{Notations and acknowledgments}\ 
We use the convention $f \ll g$ to mean $f = O(g)$; if the implied constant depends upon another variable $\eps$, we shall write $f \ll_\eps g$. $x \land y$ denotes the minimum of $x$ and $y$, and $x^+ = \max(x, 0)$ is the positive part of $x$. $\log \zeta$ is defined in the usual way (see \cite{Sel1946} for example).

The author wishes to thank Paul Bourgade for his many helpful comments.

\section{Overview of the proof of Theorem \ref{mainthm}}
We must prove two points in order to establish convergence in distribution:
(i)
    convergence of finite-dimensional distributions;
(ii)  tightness of our process $Z^{(T)}$.

\subsection{Convergence of finite-dimensional distributions.}\ 
In this section, we prove the following theorem, as a first step towards establishing Theorem \ref{mainthm}.  For complex variables, we define $ \Cov(X,Y)=\mathbb{E}[{\bar X}Y]-\mathbb{E}[{\bar X}]\mathbb{E}[Y]$

\begin{theorem}
\label{fidi}
    Let $0 \leq \alpha_1, \ldots, \alpha_n$: for $T > 0$, if $\tau$ is a uniform random variable on $[0, T]$,

    \begin{equation}
    \frac{1}{\sqrt{\log \log T}}\left(\log \zeta(\hf + \dfrac{1}{(\log T)^{\alpha_1}} + i \tau  ), \ldots, \log \zeta(\hf + \dfrac{1}{(\log T)^{\alpha_n}} + i \tau) \right)
\end{equation}
converges in law to a complex centred Gaussian vector $(Y_1, \ldots, Y_n)$, with covariances

    \begin{equation}
        \Cov(Y_i, Y_j) = 1 \land \alpha_i \land \alpha_j.
    \end{equation}
\end{theorem}

In order to achieve this, we will make use of the following lemma from \cite{bourgade}:

\begin{lemma}
\label{fidilemma}
    Let $a_{p, T}$ be complex numbers indexed by prime $p$ and $T \geq 1$. Assume that:
\begin{enumerate}

    \item $\sup_{p} |a_{p, T}| \underset{T \to \infty}{\longrightarrow} 0$;
    \item $\sum_p |a_{p, T}|^2 \underset{T \to \infty}{\longrightarrow} a^2$ for some $a \geq 0$;
    \item there exists $(m_T)$ such that $\log m_T = o(\log T)$ and
    \begin{equation}
        \sum_{p > m_T} |a_{p, T}|^2 (1+\dfrac{p}{T})\underset{T \to \infty}{\longrightarrow} 0.
    \end{equation}
\end{enumerate}
Then, if $\tau \in [0, T]$ is a uniform random variable,
    $\sum_{p} a_{p, T} p^{-i \tau}$ converges in distribution to a complex normal variable $\mathscr{N}(0, a^2)$.
\end{lemma}

In order to make use of this, we shall replace $\log \zeta$ with a related Dirichlet series. Indeed, it is known that, if $\sigma \geq \hf$,

\begin{equation}
    \log \zeta (\sigma + i \tau) - \sum_{p \leq T} \frac{1}{p^{\sigma + i \tau}}
\end{equation}
is bounded in $\mathcal{L}^2$ uniformly in $T$, and thus converges in distribution to $0$ once divided by $\sqrt{\log \log T}$. Here, we shall take 
\begin{equation}
    \sigma_i = \hf + \frac{1}{(\log T)^{\alpha_i}} \text{ for }1 \leq i \leq n.
\end{equation}
Using the Cramér–Wold method, in order to show Theorem \ref{fidi}, it is thus sufficient to show that for all $\mu_1, \ldots, \mu_n \in \C$,
\begin{equation}
    \frac{1}{\sqrt{\log \log T}}\sum_{l=1}^n \mu_l \log \zeta(\sigma_l + i \tau) \overset{\text{law}}{\longrightarrow} \mathscr{N}(0, a^2)
\end{equation}
\begin{equation*}
    \text{i.e. } \frac{1}{\sqrt{\log \log T}} \sum_{l = 1}^n \mu_l \sum_{p \leq T} \frac{1}{p^{\sigma_l + i \tau}}  \overset{\text{law}}{\longrightarrow} \mathscr{N}(0, a^2)
\end{equation*}
where $a^2 = \dsum_{1 \leq i, j \leq n} \mu_i \overline{\mu_j} (1 \land \alpha_i \land \alpha_j)$. Now, setting
\begin{equation}
    a_{p, T} = \dfrac{\mathbbm{1}_{p \leq T}}{\sqrt{\log \log T}} \sum_{l = 1}^n \dfrac{\mu_l}{p^{\sigma_l}}
\end{equation}
we simply need to check that the prerequisites of Lemma \ref{fidilemma} hold. $(i)$ is clearly true; $(iii)$ holds if we set $m_T = T^{\frac{1}{\log \log T}}$. Finally, in order to check $(ii)$, we just need to show
\begin{equation}
    \sum_{p \leq T} \dfrac{1}{p^{\sigma_i + \sigma_j}} \sim (1 \land\alpha_i \land \alpha_j) \log \log T \text{ for any }1 \leq i, j \leq n.
\end{equation}
This is shown in Lemma 3.3 of \cite{bourgade}, and we skip the proof here: it is similar to the proof of our Lemma \ref{tech} later on. All prerequisites having been checked, Theorem \ref{fidi} is therefore proven.

\subsection{The tightness criterion.}\ 
Before proving the tightness of $Z^{(T)}$, let us state a criterion which will be crucial to the proof. It is a modified version of a statement by Prokhorov \cite[Theorem 2.1]{prokh}, which itself is adapted from a criterion by Kolmogorov for the continuity of stochastic processes.

\begin{theorem}[Kolmogorov tightness criterion]
\label{kolmo}
Let $\{Z^{(T)}, T \geq 0\}$ be a sequence of stochastic processes on $\mathcal{C}([0, 1], \R)$. Assume that, for some $T_0 \geq 0$,
    \begin{itemize}
        \item $\{Z^{(T)}(x_0), T \geq T_0\}$ is tight for some $x_0$;
        \item For any $\eps > 0$, there exist events $A_\eps^T$ of probability at least $1 - \eps$, and constants  $A \geq 0, B > 1$ such that for all $T \geq T_0$ and $0 \leq a, b \leq 1$,
\begin{equation}
\label{prereq}
    \E[|Z^{(T)}(a) - Z^{(T)}(b)|^A \mathbbm{1}_{A_\eps^T}] \ll_{\eps} |a - b|^B.
\end{equation}
\end{itemize}
    Then the sequence $\{Z^{(T)}, T \geq 0\}$ is tight.
\end{theorem}

The original statement of this theorem does not include $\mathbbm{1}_{A_\eps^T}$; this is a relatively minor change, but we nevertheless include a proof for completeness.

\begin{proof}
     Setting $\eps > 0$, we must show that for large enough $T$, $Z^{(T)}$ stays in a compact subset of $\mathcal{C}^0([0, 1])$ with probability at least $1 - \eps$. Replacing $\eps$ by $2 \eps$, we may replace $Z^{(T)}$ by $Z^{(T)}\mathbbm{1}_{A_\eps^T}$, and so our main hypothesis becomes
\begin{equation}
    \forall x, y \in [0, 1], \E[|Z^{(T)}(x)-Z^{(T)}(y)|^A ] \leq C|x-y|^B.
\end{equation}
If $0 < \gamma < 1$, we are going to bound the $\gamma$-Hölder norm of $Z^{(T)}$:
\begin{equation}
    \| \phe \|_\gamma = |\phe(x_0)| + \sup_{0 \leq x, y \leq 1} \frac{|\phe(x) - \phe(y)|}{|x-y|^\gamma}.
\end{equation}
In fact, it is sufficient to bound its restriction to dyadic intervals
\begin{equation}
    \| \phe \|_\gamma^\mathcal{D} = |\phe(x_0)| + \sup_{\underset{0 \leq k \leq 2^n - 1}{n \geq 1}} \frac{|\phe(\frac{k+1}{2^n}) - \phe(\frac{k}{2^n})|}{(\frac{1}{2^n})^\gamma}
\end{equation}
since $\| \phe \|_\gamma \leq 2(1-2^{-\gamma}) \| \phe \|_\gamma^\mathcal{D} $.
Accordingly, if $n > 0$, $0 \leq k < 2^n$, and $T > T_0, M > 0$,
\begin{equation}
            \Proba(|Z^{(T)}(\frac{k+1}{2^n}) - Z^{(T)}(\frac{k}{2^n})| > \frac{M}{2^{\gamma n}})  \leq \frac{2^{A \gamma n}}{M^A} \E[|Z^{(T)}(\frac{k+1}{2^n}) - Z^{(T)}(\frac{k}{2^n})|^A]
            \leq \frac{1}{M^A} 2^{(A \gamma - B)n}.
\end{equation}
Furthermore, we can find $M' > 0$ such that $\Proba(|Z^{(T)}(x_0)| > M') < \frac{\eps}{2}$. Thus, if we take $\gamma < \dfrac{B - 1}{A}$, we may sum over all dyadic numbers:
\begin{equation}
            \Proba(\| Z^{(T)} \|_\gamma^{\mathcal{D}} > M + M') \leq \dfrac{\eps}{2} + \sum_{k, n} \Proba(|Z^{(T)}(\frac{k+1}{2^n}) - Z^{(T)}(\frac{k}{2^n})| > \frac{M}{2^{\gamma n}}) 
             \leq \frac{1}{M^A} \frac{2^{1 + A \gamma - B}}{1 - 2^{1 + A \gamma - B}} + \dfrac{\eps}{2}.
\end{equation}
For large enough $M$, this is at most $\eps$. Thus, we have shown that $\|Z^{(T)}\|_\gamma$ is bounded by a certain constant with probability at least $1 - \eps$: since the unit ball for $\|\cdot \|_{\gamma}$ is compact, this concludes our proof.
\end{proof}
If we want to extend Theorem \ref{mainthm} to $a \in [0, \infty)$ (as in Remark \ref{exttoline}), the criterion above does not apply as-is. Instead, we will note that 

\begin{equation}
    \E[\sup_{a \geq 1} |Z^{(T)}(a) - Z^{(T)}(1)|^4 \mathbbm{1}_{A_\eps^T}] \ll_\eps \frac{1}{(\log \log T)^2}
\end{equation}

via a straightforward modification of (\ref{lasteq}). This automatically guarantees convergence in law to our process $(\tilde{B}_\alpha)$ defined in (\ref{defbtilde}).

\subsection{Structure of the proof of tightness}.\ 
Theorem \ref{kolmo}  allows us to transform a tightness problem, which would require some pretty strong uniform controls on $\log \zeta$, into a moments calculation for which we have much better tools. In our case, we shall take $A = 4$ and $B=2$, owing to the roughly $\hf$-Holderian behaviour of Brownian motion.

Are the prerequisites of Theorem \ref{kolmo} verified by the process $Z^{(T)}$? The first one clearly is (taking for instance $x_0 = 0$) but the second one is much less obvious, and the purpose of the rest of this paper will be to prove that it holds. Specifically, from now on, we will set $\eps > 0$, and $0 \leq a < b \leq 1$: for large enough $T > 0$, our aim is to construct an adequate $A_\eps^T$ (independent of $a, b$) such that
\begin{equation}
\label{momenttoshow}
    \E[|Z^{(T)}(a) - Z^{(T)}(b)|^4 \mathbbm{1}_{A_\eps^T}] \leq C_\eps |a - b|^2.
\end{equation}
We will also set
\begin{equation}
    \sigma_1 = \hf + \frac{1}{(\log T)^a} \text{ and } \sigma_2 = \hf + \dfrac{1}{(\log T)^b}
\end{equation}
so that (\ref{momenttoshow}) becomes
\begin{equation}
    \E [|\log \zeta(\sigma_1 + i \tau) - \log \zeta(\sigma_2 + i \tau)|^4 \mathbbm{1}_{A_\eps^T}] \leq C_\eps(b - a)^2(\log \log T)^2.
\end{equation}
In order to show this, we will proceed by approximating $\log \zeta$ by a well-chosen Dirichlet sum. This allows us to effectively compute moments by expanding out the sum. Specifically, we will make use of the following decomposition from Selberg's original paper on the CLT \cite{Sel1946}: for $x > 1$, we may write
\begin{equation}
    \dfrac{\zeta'}{\zeta}(s) = -\sum_{n \leq x^3} \frac{\Lambda_x(n)}{n^{s}} +  e_x(s)
\end{equation} 
where 
\begin{equation}
    \Lambda_x(n) = \begin{cases}
    \Lambda(n) & \text{ if } n \leq x \\
    \Lambda(n) \dfrac{\log^2 \frac{x^3}{n} - 2 \log^2 \frac{x^2}{n}}{2 \log^2 x} & \text{ if } x \leq n \leq x^2 \\
    \Lambda(n) \dfrac{\log^2 \frac{x^3}{n}}{2 \log^2 x} & \text{ if } x^2 \leq n \leq  x^3 
\end{cases}. 
\end{equation}

\noindent Here, we will be taking $x = T^{\frac{1}{20}}$, although all of the following results are valid for $x = T^{c}$ with small enough $c$. As a result,
\begin{equation}
\label{termstobound}
    \begin{split}
        \E[|Z^{(T)}(a) - Z^{(T)}(b)|^4 \mathbbm{1}_{A_\eps^T}] & \leq \dfrac{32}{(\log \log T)^2}\E \left[\left|\sum_{n \leq x^3} \frac{\Lambda_x(n)}{n^{\sigma_1 + i \tau} \log n} - \sum_{n \leq x^3} \frac{\Lambda_x(n)}{n^{\sigma_2 + i \tau} \log n} \right|^4 \right] \\
        & + \dfrac{32}{(\log \log T)^2} \E \left[\left|\int_{\sigma_2}^{\sigma_1}e_x(\sigma + i \tau) d \sigma \right|^4\mathbbm{1}_{A_\eps^T}\right].
    \end{split}
\end{equation}
The first term can be bounded quite effectively; this will be done in Section 4. To bound the second term, we will rely upon methods developed by Selberg in \cite{Sel1946}; this will be done in Section 5. In particular, the specific choice of $\Lambda_x$ was made in order to be able to apply said methods.  

With both terms bounded, we will mostly have proven our main theorem. However, setting $\sigma_c =  \dfrac{1}{2} + \dfrac{40 \log \frac{1}{\eps}}{\log T}$, it turns out that this method breaks down when $\sigma_2 \leq \sigma_c$. This case is not too complicated, and is handled separately at the end of the paper.

\section{Moments of Dirichlet sums}

As announced, in this section we will show the following proposition.

\begin{proposition}
\label{lemma}
Here and for the rest of this paper, we will take $x = T^{\frac{1}{20}}$. Then,
    \begin{equation}
        \E \left[\left|\sum_{n \leq x^3} \frac{\Lambda_x(n)}{n^{\sigma_1 + i \tau} \log n} - \sum_{n \leq x^3} \frac{\Lambda_x(n)}{n^{\sigma_2 + i \tau} \log n} \right|^4 \right] \ll (b-a)^2 (\log \log T)^2.
    \end{equation}
\end{proposition}

In order to prove Proposition \ref{lemma}, we shall use the two following technical lemmas:

\begin{lemma}
    Let $\phe : \N \rightarrow \R$ satisfy the following conditions:
    \label{lemmaphe}
    \begin{itemize}
        \item $\phe(n) = 0$ if $n$ is not a $p^k$ for some prime $p$ and $k > 0$;
        \item if $p$ is prime and $i \geq 1$, $|\phe(p^i)| \leq i |\phe(p)|$
    \end{itemize}
    Then,
\begin{equation}
    \E\left[ \left|\sum_{n \leq x^3} \frac{\phe(n)}{n^{\hf + i \tau}} \right|^4 \right] \ll \left( \sum_{p \leq x^3 \text{ prime}} \frac{\phe(p)^2}{p} \right)^2,
\end{equation}
    with the implied constant being independent of $\phe$.
\end{lemma}

\begin{lemma}
\label{tech}
    Let $0 \leq \alpha \leq \beta \leq 1$, be functions of $T$, and set  $\eta = (\log T)^{-\alpha}$, $\eta' = (\log T)^{-\beta}$. Then,
    \begin{equation}
            \sum_{p \leq x^3 \text{ prime}} \frac{1}{p^{1 + \eta}} - \frac{1}{p^{1 + \eta'}} \ll (\beta - \alpha)\log \log T.
    \end{equation}
\end{lemma}

These lemmas are very similar to existing results in the literature, but those are not quite sufficient for our purposes due to the dependence on $\sigma, \sigma'$.

Assuming these lemmas, we may set
\begin{equation}
    \phe(n) = \frac{\Lambda_x(n)}{\log n}(n^{-(\sigma_1 - \hf)} - n^{-(\sigma_2 - \hf)})
\end{equation}
and, applying Lemma \ref{lemmaphe},
\begin{equation}
    \begin{split}
        \E \left[\left|\sum_{n \leq x^3} \frac{\Lambda_x(n)}{n^{\sigma_1 + i \tau} \log n} - \sum_{n \leq x^3} \frac{\Lambda_x(n)}{n^{\sigma_2 + i \tau} \log n} \right|^4 \right] & \ll 
        \left( \sum_{p \leq x^3 \text{ prime}} (p^{-\sigma_1} - p^{-\sigma_2})^2 \right)^2 \\
    \end{split}
\end{equation}
\begin{equation*}
    = \left( \sum_{p \leq x^3} \left(\frac{1}{p^{2\sigma_1}} - \frac{1}{p^{\sigma_1 + \sigma_2}}\right) - \left( \frac{1}{p^{\sigma_1 + \sigma_2}} - \frac{1}{p^{2 \sigma_2}} \right) \right)^2.
\end{equation*}
Applying Lemma \ref{tech}, this shows Proposition \ref{lemma}.

\begin{proof}[Proof of Lemma \ref{lemmaphe}]
We may write
\begin{equation}
    \E\left[ \left|\sum_{n \leq x^3} \frac{\phe(n)}{n^{\hf + i \tau}} \right|^4 \right]   = \E\left[ \left|  \sum_{m, n \leq x^3} \dfrac{\phe(m)\phe(n)}{\sqrt{mn}} e^{-i\tau \log(mn)}\right|^2 \right]
 = \E \left[ \left|\sum_{l \geq 1} e^{-i \tau \log l}\sum_{\underset{mn = l}{m, n \leq x^3}} \dfrac{\phe(m)\phe(n)}{\sqrt{mn}} \right|^2 \right]
\end{equation}
in order to apply the following identity by Montgomery-Vaughan.

\begin{lemma}
    Let $\lambda_1, \ldots, \lambda_N \in \R$, $\alpha_1, \ldots, \alpha_N$ and set $\delta = \ds\min_{i, j \leq N}|\lambda_i - \lambda_j|$. Then,

    \begin{equation}
        \int_{0}^T \left| \sum_{k=1}^N \alpha_k e^{i \lambda_k t} \right|^2 dt = (T + O(\delta^{-1})) \sum_{k=1}^{N}|\alpha_k|^2.
    \end{equation}
\end{lemma}
A proof of this can be found in \cite{hilbert}. In our case, we obtain
\begin{equation}
    \E\left[ \left|\sum_{n \leq x^3} \frac{\phe(n)}{n^{\hf + i \tau}} \right|^4 \right]  \ll \sum_{l \geq 1} \dfrac{1}{l} \left| \sum_{\underset{mn = l}{m, n \leq x^3}} {\phe(m)\phe(n)}\right|^2 
     \overset{\text{def}}{=} \sum_{l \geq 1}\frac{1}{l} \Phi_x(l)^2 
     \leq \sum_{p <  q \text{ prime}}  \sum_{1 \leq i, j} \frac{\Phi_x(p^iq^j)^2}{p^iq^j} + \sum_{i \geq 2} \sum_{p \text{ prime}} \frac{\Phi_x(p^i)^2}{p^i}.
\end{equation}
We separate these two terms in order to effectively bound $\Phi_x$ in each case. To bound the first term, expanding out $\Phi_x$,
\begin{equation}
       \sum_{p <  q \text{ prime}}  \sum_{1 \leq i, j} \frac{\Phi_x(p^iq^j)^2}{p^iq^j}  \leq \sum_{p, q \leq x^3 \text{ prime}} \sum_{1 \leq i, j} \frac{\phe(p^i)^2 \phe(q^j)^2}{p^iq^j}
 \ll \sum_{p, q \leq x^3 \text{ prime}} \frac{\phe(p)^2 \phe(q)^2}{pq}
 \ll \left( \sum_{p \leq x^3 \text{ prime}} \frac{\phe(p)^2}{p} \right)^2.
\end{equation}

We handle the second term in the same manner:
\begin{multline*}
        \sum_{i \geq 2} \sum_{p \text{ prime}} \frac{\Phi_x(p^i)^2}{p^i}  = \sum_{i \geq 2} \sum_{p \text{ prime}} \sum_{k+l = i, k'+l'=i} \frac{\phe(p^k)\phe(p^{k'})\phe(p^l)\phe(p^{l'})}{p^i} 
         \leq \sum_{p \text{ prime}} \phe(p)^4 \sum_{i \geq 2} \frac{i^6}{p^i} \\
         \ll \sum_{p \text{ prime}} \frac{\phe(p)^4}{p^2} \leq \left( \sum_{p \leq x^3 \text{ prime}} \frac{\phe(p)^2}{p} \right)^2.
    \end{multline*}
This concludes the proof of Lemma \ref{lemmaphe}.
\end{proof}
We now just have to prove our other technical result.

\begin{proof}[Proof of Lemma \ref{tech}]
First, note that by taking the derivative of $\alpha \mapsto \dfrac{1}{(\log T)^\alpha}$, we obtain

\begin{equation}
    \label{iaf}
    \eta - \eta' \ll \eta (\beta - \alpha) \log \log T.
\end{equation}
We shall be using equation (\ref{iaf}) throughout the rest of this paper.
Let us rewrite:
\begin{equation}
        \sum_{p \leq x^3} \frac{1}{p^{1 + \eta}}  = \sum_{n \leq x^3} \frac{\pi(n) - \pi(n-1)}{n^{1+\eta}} 
 = \sum_{n \leq x^3} \pi(n) (\frac{1}{n^{1 + \eta}} - \frac{1}{(n+1)^{1 + \eta}}) + \frac{\pi(N)}{N^{1+\eta}}.
\end{equation}
where $N = \lfloor x^3 \rfloor$, and $\pi$ denotes the prime-counting function. We now wish to approximate this sum by its associated integral. Specifically, setting

\begin{equation}
    I_x(\eta) = \int_1^{x^3} \pi(u)(\frac{1}{u^{1+\eta}} - \frac{1}{(u+1)^{1 + \eta}})du,
\end{equation}
we will split up our problem:
\begin{equation}
    \sum_{p \leq x^3} \frac{1}{p^{1 + \eta}} - \sum_{p \leq x^3} \frac{1}{p^{1 + \eta'}} = \left(I_x(\eta')  - \sum_{p \leq x^3} \frac{1}{p^{1 + \eta'}} \right) - \left(I_x(\eta)  - \sum_{p \leq x^3} \frac{1}{p^{1 + \eta}} \right) + (I_x(\eta) - I_x(\eta')).
\end{equation}
Now,
\begin{equation}
    \begin{split}
        \dfrac{d}{d\eta} & \left(I_x(\eta)  - \sum_{p \leq x^3} \frac{1}{p^{1 + \eta}} \right) \\
        & = (1 + \eta) \left( \int_{1}^{x^3}\pi(u) (-\frac{\log u}{u^{1 + \eta}} + \frac{\log \lfloor u \rfloor}{\lfloor u\rfloor^{1 + \eta}} + \frac{\log (u+1)}{(u+1)^{1+\eta}} - \frac{\log \lfloor u+1 \rfloor}{\lfloor u+1\rfloor ^{1+\eta}}) du - \frac{\pi(N) \log N}{N^{1 + \eta}} \right) \\
        & \ll \int_{1}^{x^3}\pi(u) \frac{\log u}{u^{3+\eta}}du + \frac{\pi(N)\log N}{N^{1+\eta}} \ll 1
    \end{split}
\end{equation}
since $\pi(u) \sim \dfrac{u}{\log u}$. As a result,
\begin{equation}
    \left(I_x(\eta)  - \sum_{p \leq x^3} \frac{1}{p^{1 + \eta}} \right) - \left(I_x(\eta')  - \sum_{p \leq x^3} \frac{1}{p^{1 + \eta'}} \right) \ll (\eta - \eta') \ll (\beta - \alpha) \log \log T
\end{equation}
by (\ref{iaf}). As a result, we now just need to control $I_x(\eta) - I_x(\eta')$:
\begin{equation}
\begin{split}
    I_x(\eta) - I_x(\eta') & \ll \int_{2}^{x^3} \dfrac{u}{\log u}\left(\frac{1}{u^{1 + \eta}} - \frac{1}{(u+1)^{1 + \eta}} - \frac{1}{u^{1 + \eta'}} + \frac{1}{(u+1)^{1 + \eta'}} \right) du \\
    & = \int_2^{x^3} \left(\frac{\eta}{u^{1 + \eta} \log u} - \frac{\eta'}{u^{1 + \eta'}\log u} \right) du + \int_2^{x^3} \eps(u, \eta) - \eps(u, \eta')du
\end{split}
\end{equation}

with $\eps(u, \eta) = \dfrac{1}{u^{\eta} \log u} - \dfrac{u}{(u+1)^{1 + \eta} \log u} - \dfrac{\eta}{u^{1 + \eta} \log u}$. Since $\dfrac{d}{d \eta} \eps(u, \eta) \ll \dfrac{1}{u^{2 + \eta}}$,

\begin{equation}
    \int_2^{x^3} (\eps(u, \eta) - \eps(u, \eta'))du \ll \eta - \eta' \ll (\beta - \alpha) \log \log T.
\end{equation}
In order to bound the main integral, we may now set $v = \eta \log u$ (resp. $v = \eta' \log u$):
\begin{equation}
    \begin{split}
        \int_2^{x^3} & \left(\frac{\eta}{u^{1 + \eta} \log u} - \frac{\eta'}{u^{1 + \eta'}\log u} \right) du  = \eta \int_{\eta \log 2}^{3 \eta \log x} \frac{dv}{v e^v} - \eta' \int_{\eta' \log 2}^{3\eta' \log x} \frac{dv}{v e^v} \\
        & = (\eta - \eta') \int_{\eta \log 2}^{3 \eta' \log x} \dfrac{dv}{v e^v} + \eta \int_{3 \eta' \log x}^{3 \eta \log x}\dfrac{dv}{v e^v} - \eta' \int_{\eta' \log 2}^{\eta \log 2} \dfrac{dv}{v e^v} \\
        & = I_1 + I_2 - I_3.
    \end{split}
\end{equation}
It is quite clear that $I_2 \ll I_3$; meanwhile,
\begin{equation}
    I_3 = \eta' \int_{\eta' \log 2}^{\eta \log 2} \frac{1}{v}(1 + O(1))dv \ll \eta' \log \frac{\eta}{\eta'} + (\eta - \eta') \ll (\beta - \alpha) \log \log T.
\end{equation}
Finally,
\begin{equation}
        I_1  \ll (\eta - \eta') \left( \int_{\eta \log 2}^1 \dfrac{dv}{v} + \int_{1}^{3 \eta' \log x} e^{-v} dv \right) 
\ll \eta (1 + \log \eta)(\beta - \alpha) \log \log T \ll (\beta - \alpha)(\log \log T),
\end{equation}
which concludes the proof.
\end{proof}

\section{The contribution of zeta zeroes}

As a reminder, we have set $\sigma_c = \Hf + \dfrac{40 \log \frac{1}{\eps}}{\log T}$. We still have two points to handle in order to apply Theorem \ref{kolmo}:
(i) bounding the second term in (\ref{termstobound});
(ii) handling the case $\sigma_2 \leq \sigma_c$.

In the interests of legibility, we will define (and work with) $\eta_1 = \sigma_1 - \Hf$, $\eta_2 = \sigma_2 - \Hf$, $\eta_c = \sigma_c - \Hf$, etc.

\subsection{Bounding the error $e_x$. }

Recall that we set

\begin{equation}
    e_x(s) = \frac{\zeta'}{\zeta}(s) + \sum_{n \leq x^3} \frac{\Lambda_x(n)}{n^s}.
\end{equation}
We wish to show that
\begin{equation}
    \label{toshow}
   \E \left[\left|\int_{\sigma_2}^{\sigma_1}e_x(\sigma + i \tau) d \sigma \right|^4\mathbbm{1}_{A_\eps^T}\right] \ll_\eps (b-a)^2(\log \log T)^2
\end{equation}
when $\sigma_c \leq \sigma_2 \leq \sigma_1$. Our main tool for showing this will be the following identity from \cite[equation (4.9)]{Sel1946}

\begin{lemma}
\label{selblem}
    Let $t \geq 2$ and $2 \leq x \leq t^2$.

Furthermore, set

$$\sigma_{x, t} = \Hf + 2 \max_\rho (\beta , \frac{2}{\log x})$$

where $\rho = \Hf + \beta + i \gamma$ ranges over all zeroes of $\zeta$ such that $|t - \gamma| \leq \dfrac{x^{3|\beta|}}{\log x}$.
Then, if $\sigma \geq \sigma_{x, t}$,

\begin{equation}
    e_x(\sigma + i t) \ll  x^{- \frac{1}{2}(\sigma - \hf)}  \left( \left| \sum_{n \leq x^3} \frac{\Lambda_x(n)}{n^{\sigma_{x, t} + it}} \right| + \log t \right).
\end{equation}

\end{lemma}

Applying this lemma, if $\sigma_{x, t} \leq \sigma_2 \leq \sigma_1$ and $T^{\frac{1}{2}} \leq t \leq T$,

\begin{equation}
    \begin{split}
        \left|\int_{\sigma_2}^{\sigma_1}e_x(\sigma + i t) d \sigma \right| 
        & \ll \left( \left|\sum_{n \leq x^3} \frac{\Lambda_x(n)}{n^{ \sigma_{x, t} + it}} \right| + \log T  \right) \min(\frac{x^{-\frac{\eta_1}{2}}}{\log x}, (\sigma_1 - \sigma_2) x^{-\frac{\eta_2}{2}}).
    \end{split}
\end{equation}

This is a good start, but how do we go from the condition $\sigma_2 \geq \sigma_{x, t}$ to $\sigma_2 \geq \sigma_c$? For this, we need to show that $\sigma_{x, t}$ is usually smaller than $\sigma_c$: we will then cut out the region where $\sigma_{x, t} \geq \sigma_c$ by choosing $A_\eps^T$ adequately. This is the object of the following lemma.
\begin{lemma}
\label{cutset}
Set
    \begin{equation}
    Y_{\eps} = \bigcup_{\rho} [\gamma - \eps\frac{x^{4 |\beta|}}{\log x}, \gamma + \eps\frac{x^{4|\beta|}}{\log x}]
\end{equation}

where $\rho = \hf + \beta + i \gamma$ ranges over non-trivial zeroes of $\zeta$, then:
\begin{itemize}
    \item the measure of $Y_{\eps} \cap [0, T]$ is $O(\eps T)$;
    \item if $t \in  [0, T] \setminus Y_\eps$, $\sigma_{x, t} \leq \sigma_c$.
\end{itemize}
\end{lemma}

\begin{proof}
    We may ignore the zeroes to the left of the critical axis, since $\zeta$ has reflectional symmetry with regards to the critical axis.
    
    Set, for $\eta \geq 0$,

    \begin{equation}
        \mathcal{N}(\eta, T) = \# \{ \rho = \Hf + \beta + i \gamma \text{ such that } \zeta(\rho) = 0, \beta \geq \eta, 0 \leq \gamma \leq T \}.
    \end{equation}
    It is known (see \cite{Sel1946}) that $ \mathcal{N}(\eta, T) \ll T \log T \exp(-\frac{1}{4} \eta \log T)$. Thus, setting $N = \mathcal{N}(0, T)$, we may label $\hf + \beta_1 \geq \hf + \beta_2 \geq \ldots \geq \hf + \beta_N$ the abscissae of zeroes of $\zeta$ with ordinates in $[0, T]$. Now:
    \begin{equation}
    \begin{split}
            |Y_\eps \cap [0, T]|  \leq \sum_{i = 1}^{N - 1} 2 \eps \frac{x^{4 \beta_i}}{\log x}
  & = \frac{1}{2} \eps \sum_{i=1}^{N-1} \mathcal{N}(\beta_i, T)\int_{\beta_{i+1}}^{\beta_i} x^{4 \beta} d\beta + 2 \eps N \frac{x^{4 \beta_N}}{\log x}
  \\
  & \leq \eps \int_0^{1} x^{4 \beta} \mathcal{N}(\beta, T) d\beta + 2 \eps N \frac{x^{4 \beta_N}}{\log x}.
  \end{split}
    \end{equation}
   Since $\beta_N \ll \dfrac{1}{\log T}$ and $N \ll T \log T$, $2 \eps N \dfrac{x^{4 \beta_N}}{\log x} = O(\eps T)$. Meanwhile,
    \begin{equation}
            \int_0^{1} x^{4 \beta} \mathcal{N}(\beta, T) d\beta  \ll T \log T\int_0^1 T^{-\frac{\beta}{20}} d \beta 
             \ll T.
    \end{equation}
    This shows the first part of Lemma \ref{cutset}. For the second point, simply note that, if $t \in [0, T] \setminus Y_{\eps}$, and $\rho = \hf + \beta + i \gamma$ is a zero of $\zeta$ such that $|t - \gamma| \leq \frac{x^{3 \beta}}{\log x}$,

    \begin{equation}
        \eps \frac{x^{4 \beta}}{\log x} \leq |t - \gamma| \leq \frac{x^{3 \beta}}{\log x} \text{ so } x^{\beta} \leq \frac{1}{\eps} \text{, and } \beta \leq \frac{\log \frac{1}{\eps}}{\log x}.
    \end{equation}
    Taking the maximum over all applicable $\rho$, we see that $\sigma_{x, t} \leq \Hf + \dfrac{40 \log \frac{1}{\eps}}{\log T}$.
\end{proof}

This means that we can set $A_\eps^T = ( \tau \notin Y_\eps, \tau > \sqrt{T})$ and apply Lemma \ref{selblem} to tackle (\ref{toshow}):
\begin{equation}
    \begin{split}
        \E\left[\left|\int_{\sigma_2}^{\sigma_1}e_x(\sigma + i \tau) d \sigma\right|^4\mathbbm{1}_{A_\eps^T}\right] & \ll \left( \E \left[ \left|\sum_{n \leq x^3} \frac{\Lambda_x(n)}{n^{\sigma_{x, \tau} + i\tau}} \right|^4 \right] + (\log T)^4\right)  \min(\frac{x^{-2\eta_1}}{(\log x)^4}, (\eta_1 - \eta_2)^4 x^{-2\eta_2}).
    \end{split}
\end{equation}

To conclude, we will use the following proposition, whose proof will be given shortly:

\begin{proposition} We have
\label{prop1}
$$\E \left[ \left| \sum_{n \leq x^3} \frac{\Lambda_x(n)}{n^{\sigma_{x, \tau} + i\tau}} \right|^4\mathbbm{1}_{t \notin Y_{\eps}}\right] \ll_\eps (\log T)^4.$$
\end{proposition}

Assuming Proposition \ref{prop1}, we can consider two cases.
First,  if $b - a \leq \dfrac{1}{\log \log T}$, then $\dfrac{\eta_1}{\eta_2} \ll 1$ and, applying (\ref{iaf}),
    \begin{multline*}
            \E[\int_{\sigma_2}^{\sigma_1}|e_x(\sigma + i \tau)|^4 d \sigma\mathbbm{1}_{A_\eps^T}] 
             \ll_\eps (\log T)^4 \eta_1^4 (b-a)^4 (\log \log T)^4 x^{-2\eta_2} \\
             \ll (b-a)^4 (\eta_2 \log T)^4 e^{-\frac{1}{10} \eta_2 \log T}(\log \log T)^4 
             \ll (b-a)^2 (\log \log T)^2.
    \end{multline*}
If $b-a \geq \dfrac{1}{\log \log T}$, then
    \begin{equation}
            \E[\int_{\sigma_2}^{\sigma_1}|e_x(\sigma + i \tau)|^4 d \sigma\mathbbm{1}_{A_\eps^T}] 
             \ll_\eps \left(\frac{\log T}{\log x}\right)^4 x^{-2 \eta_2} 
             \ll 1  \ll (b-a)^2 (\log \log T)^2.
    \end{equation}
We therefore just need to show Proposition \ref{prop1}.
\begin{proof}

Set $\eta_0 = \dfrac{1}{\log T}$: then,
\begin{equation}
\label{uhhhhh}
    \E \left[ \left| \sum_{n \leq x^3} \frac{\Lambda_x(n)}{n^{\sigma_{x, \tau} + i\tau}} \right|^4\mathbbm{1}_{A_\eps^T}\right] \leq 8\E \left[ \left| \sum_{n \leq x^3} \frac{\Lambda_x(n)}{n^{\hf + \eta_0 + i\tau}} \right|^4\right] + 8\E \left[ \left| \sum_{n \leq x^3} \frac{\Lambda_x(n)}{n^{\hf + i\tau}} \left(n^{-\eta_0} - n^{-\eta_{x, \tau}} \right) \right|^4\mathbbm{1}_{A_\eps^T}\right].
\end{equation}
However, applying Lemma \ref{lemmaphe} with $\phe(n) = \Lambda_x(n) n^{-\eta_0}$:
\begin{equation}
    \E \left[ \left| \sum_{n \leq x^3} \frac{\Lambda_x(n)}{n^{\hf + \eta_0 + i\tau}} \right|^4\right]  \ll \left( \sum_{p \leq x^3 \text{ prime}} \frac{(\log p)^2}{p^{1 + 2\eta_0}} \right)^2
     \ll \left(\int_2^{A} \frac{(\log t)^{1-2 \eta_0}}{t^{1 + 2 \eta_0}} dt \right)^2
\end{equation}
for some $A \sim \dfrac{x^3}{3 \log x}$. Setting $u = \eta_0 \log t$ and changing variables,
\begin{equation}
        \int_2^{A} \frac{(\log t)^{1-2 \eta_0}}{t^{1 + 2 \eta_0}} dt  = \dfrac{1}{\eta_0^{2-2\eta_0}} \int_{\eta_0 \log 2}^{\eta_0 \log A} \frac{u^{1-2\eta_0}}{e^{2u}}du
 \ll \eta_0^{-2}.
\end{equation}
 Meanwhile, if we look at the second term in (\ref{uhhhhh}), applying Hölder's inequality:
\begin{equation*}
    \begin{split}
         \E \left[ \left| \sum_{n \leq x^3} \frac{\Lambda_x(n)}{n^{\hf + i\tau}} \left(n^{-\eta_0} - n^{-\eta_{x, \tau}} \right) \right|^4\mathbbm{1}_{A_\eps^T}\right] 
         & \ll \E \left[ \left| \int_{\eta_0}^{\eta_{x, \tau}} \sum_{n \leq x^3} \frac{\Lambda_x(n) \log n}{n^{\hf + \eta + i\tau}} d \eta \right|^4\mathbbm{1}_{A_\eps^T}\right] \\
         & \leq (\eta_c - \eta_0)^3 \int_{\eta_0}^{\eta_c}  \E \left[ \left| \sum_{n \leq x^3} \frac{\Lambda_x(n) \log n}{n^{\hf + \eta + i\tau}} \right|^4 \mathbbm{1}_{A_\eps^T} \right] d \eta \\
         & \ll_\eps \eta_0^{-4} \ll (\log T)^{4}
    \end{split}
\end{equation*}
by the same reasoning, applying Lemma \ref{lemmaphe}. This gives the expected result.
\end{proof}

\subsection{Towards the critical line}.\ 
As mentioned earlier, the case $\sigma_2 \leq \sigma_c$ needs to be handled separately. Specifically, it still remains to be shown that:
\begin{proposition}
\label{useabove}
    If $\sigma_2 \leq \sigma_1 \leq \sigma_c$,
\begin{equation}
    \E[|\log \zeta (\sigma_1 + i \tau) - \log \zeta (\sigma_2 + i \tau)|^4 \mathbbm{1}_{A_\eps^T}] \ll_\eps (b-a)^2(\log \log T)^2.
\end{equation}
\end{proposition}

The case where $\sigma_2 \leq \sigma_c \leq \sigma_1$ follows easily, applying the "triangular inequality" $|a+b|^4 \ll |a|^4 + |b|^4$.

\begin{proof}

Note that, if $\sigma \leq \sigma_c,$ and $0 \leq t \leq T$,

\begin{equation}
    \begin{split}
        |\frac{\zeta'}{\zeta}(\sigma + it) - \frac{\zeta'}{\zeta}(\sigma_c + it)| & = |\sum_{\rho = \beta + i \gamma} \frac{1}{\eta - \beta + i (t - \gamma)} -  \frac{1}{\eta_c - \beta + i (t - \gamma)} + O(\log T)| \\
        & \leq \sum_{\rho} \frac{\eta_c - \eta}{|\eta - \beta + i (t - \gamma)||\eta_c - \beta + i (t - \gamma)|} + O(\log T) \\
        & \leq (\eta_c - \eta)\sum_\rho \frac{1}{(t - \gamma)^2} + O(\log T) .
    \end{split}
\end{equation}
Now, 
$$\E \left[ \sum_\rho \frac{1}{(t - \gamma)^2} \mathbbm{1}_{t \notin Y_{\eps}} \right] \ll_\eps (\log T)^2.$$

As a result, we can increase the size of $Y_\eps$ in such a way that the measure of $Y_\eps \cap [0, T]$ remains $O(\eps T)$, and if $t \in [0, T] \setminus Y_\eps$,
\begin{equation}
\label{logder}
    \sum_\rho \frac{1}{(t - \gamma)^2} \ll_\eps (\log T)^2 \text{ and so } |\frac{\zeta'}{\zeta}(\sigma + it) - \frac{\zeta'}{\zeta}(\sigma_c + it)| \ll_\eps \log T.
\end{equation}
Furthermore, by applying Lemma \ref{selblem} and adapting the proof of Proposition \ref{prop1}, we can see that
\begin{equation}
    \E[|\frac{\zeta'}{\zeta}(\sigma_c + i \tau)|^4 \mathbbm{1}_{A_\eps^T} ] \ll  \E[|e_x(\sigma_c + i \tau)|^4\mathbbm{1}_{A_\eps^T} ] + \E\left[\left|\sum_{n \leq x^3} \frac{\Lambda_x(n)}{n^{\sigma_c + it}} \right|^4\mathbbm{1}_{A_\eps^T}\right ] \ll_\eps (\log T)^4.
\end{equation}
Consequently,
\begin{equation} \label{lasteq}
    \begin{split}
        \E [ |\log \zeta (\sigma_1 + i \tau) - \log \zeta (\sigma_2 + i\tau)|^4 \mathbbm{1}_{A_\eps^T} ] & \ll \E \left[ \left|\int_{\sigma_2}^{\sigma_1} \frac{\zeta'}{\zeta}(\sigma + i\tau)d\sigma\right|^4\mathbbm{1}_{A_\eps^T}\right] \\
        & \ll_\eps (\sigma_1 - \sigma_2)^4 ((\log T)^4 + \E [|\frac{\zeta'}{\zeta} (\sigma_c + i\tau)|^4\mathbbm{1}_{A_\eps^T}]) \\
        & \ll_\eps (\eta_1 \log T)^4 (b - a)^4 (\log \log T)^4 \\
        & \ll_\eps (b - a)^2 (\log \log T)^2
    \end{split}
\end{equation}
given that $b - a \ll_\eps \dfrac{1}{\log \log T}$.
\end{proof}

\end{document}